\documentclass[10pt]{amsart}
\usepackage{palatino}
\usepackage{amsfonts}
\usepackage{amssymb}
\usepackage{amscd}
\usepackage[breaklinks,bookmarksopen,bookmarksnumbered]{hyperref}

\newtheorem{theorem}{Theorem}[section]
\newtheorem{proposition}[theorem]{Proposition}
\newtheorem{corollary}[theorem]{Corollary}
\newtheorem{lemma}[theorem]{Lemma}
\theoremstyle{definition}

\newtheorem{remark}[theorem]{Remark}

\newtheorem{conjecture/question}[theorem]{Conjecture/Question}

\newtheorem{remark/definition}[theorem]{Remark/Definition}
\newtheorem{terminology/notation}[theorem]{Terminology/Notation}
\newtheorem{assumption}[theorem]{Assumption}
\def\Aut{{\rm Aut}}
\def\Pic{{\rm Pic}}
\def\Ker{{\rm Ker}}

\title[Rationality of moduli of higher spin curves in low genus]{On the rationality of the moduli of higher spin curves in low genus}
 \author{Letizia Pernigotti}  
 \address{Universit\`a di Trento, Dipartimento di Matematica, Via Sommarive 14, 38123 Povo\hfill  
  \indent Trento, Italy}
 \email{{ pernigotti@science.unitn.it}}
  \author{Alessandro Verra}
  \address{Universit\'a Roma Tre, Dipartimento di Matematica, Largo San Leonardo Murialdo \hfill
\indent 1-00146 Roma, Italy}
 \email{{ verra@mat.uniroma3.it}}
\thanks {Supported by PRIN Project 2010-11 'Geometria delle variet\'a algebriche'  of MIUR and by GNSAGA group of INdAM.}
\begin{document}

 \maketitle
 
\vspace{0,4cm}
\begin{center}\parbox{0.8\textwidth}{
\begin{small}
ABSTRACT: The global geometry of the moduli spaces of higher spin curves and their birational classification is largely unknown for $g \geq 2$ and $r >2$. 
Using quite related geometric constructions, we almost complete the picture of the known results in genus $g\leq 4$ showing the rationality of the moduli spaces
of even and odd $4$-spin curves of genus $3$, of odd spin curves of genus $4$ and of $3$-spin curves of genus $4$.

\vspace{0,4cm}


\vspace{0,4cm}

\end{small}
}\end{center}

 
\par \noindent
\section{Introduction}

Let $C$ be a smooth, irreducible complex projective curve of genus $g$, a theta characteristic  on $C$ is a square root $\eta$
of the canonical sheaf $\omega_C$.  By definition a pair $(C, \eta)$ is a spin curve. It is said to be even or odd according to the parity of $h^0(\eta)$.  Starting
from Cornalba's paper \cite{cornalba1989moduli}, the moduli space $\mathcal S_g$ of spin curves of genus $g$ and its compactifications became object of systematic investigations. As is well known
$\mathcal S_g$ is split in two irreducible connected components $\mathcal S^+_g$ and $\mathcal S^-_g$. They respectively correspond to 
moduli of even and odd spin curves. The Kodaira dimension of $\mathcal S^{\pm}_g$ is completely known, as well as several facts about rationality or unirationality in low genus. 
The picture is as follows for even or odd spin curves:
\begin{itemize} \it
\item[$\circ$] $\mathcal S^+_g$ is uniruled for $g \leq 7$,
\item[$\circ$] $\mathcal S^+_8$ has Kodaira dimension zero,
\item[$\circ$] $\mathcal S^+_g$ is of general type for $g \geq 9$.
\item[$\circ$] $\mathcal S^-_g$ is uniruled for $g \leq 11$,
\item[$\circ$]  $\mathcal S^-_g$ is of general type for $g \geq 12$.
\end{itemize} 
Moreover the unirationality of $\mathcal S^-_g$ and $\mathcal S^+_g$ has been proved respectively  
for $g \leq 8$ and $g \leq 6$. Concerning the rationality problem, $\mathcal S^{\pm}_g$ is classically
known to be rational for $g \leq 3$, while the rationality of $\mathcal S^{+}_4$ is a recent result. For more details on the above picture see \cite{farkas2010birational}, \cite{farkas2010geometry}, 
\cite{farkas2012moduli},\cite{takagi2009moduli}, \cite{verra2013rational}. \par
Higher spin curves generalize spin curves. By definition a higher spin curve of genus $g$ and order $r$ is a pair $(C, \eta)$ such that $\eta^{\otimes r} \cong \omega_C$. We will also say that
$(C, \eta)$ is an \it $r$-spin curve \rm of genus $g$. The moduli 
spaces of these pairs are denoted by $\mathcal S^{1/r}_g$. They  were constructed by Jarvis in \cite{jarvis1998torsion} and then studied by several 
authors, see for instance \cite{caporaso2007moduli},  \cite{chiodo2008stable},\cite{jarvis2001picard}.  \par
Concerning the irreducibility of these spaces, it is useful to recall since now how they behave: $\mathcal S^{1/r}_g$ is irreducible if $r$ is odd and $g \geq 2$, while $\mathcal S^{1 / r}_g$ is split in two irreducible connected components if $r$ is even and $g \geq 2$, \cite{jarvis2000geometry}. They are distinguished by the condition that $\eta^{\otimes \frac r2}$ is an even or odd theta characteristic.  
However, with the exception of the case of genus 1, the global geometry of $\mathcal S^{1 / r}_g$  appears to be  largely unknown for $r >2$.  \par From another side a natural, elementary, remark is that for every curve $C$ the canonical sheaf $\omega_C$ not only admits square roots, but the roots of order  $2g-2$ and $g-1$ as well. Restricting to $g-1$ roots, they form configurations of line bundles of degree two which are worth of being studied.  \par For $r = g-1$ the forgetful map
$
f: \mathcal S^{1/(g-1)} \to \mathcal M_g
$
has degree $(g-1)^{2g}$. Since this grows up very fast, it is seems natural to expect that $\mathcal S^{1 / (g-1)}_g$ becomes of general type after very few exceptions. About this, assume that 
$g-1$ is even so that $\eta^{\otimes (g-1)/2}$ is a theta characteristic. Then every irreducible component of $\mathcal S^{1/(g-1)}_g$ dominates $\mathcal S^+_g$ or $\mathcal S^-_g$ if $g$ is odd, via the assignement $(C, \eta) \to (C, \eta^{\otimes (g-1)/2})$. \par  Therefore, in  view of the picture on moduli of spin curves, there exist irreducible components of  $\mathcal S^{1/(g-1)}_g$ of non negative Kodaira dimension as soon as $g \geq 8$.   In this frame the first unknown case of low genus to be considered is the genus 4 case. Somehow surprisingly this is still an exception. We prove in this note that
 \begin{theorem} The moduli space of 3-spin curves of genus 4 is rational. \end{theorem}
Let $(C, \eta)$ be a general $3$-spin curve of genus 4. The starting point for proving the theorem is the remark that giving $(C, \eta)$ is equivalent to give the unique effective
divisor $t \in \vert \eta^{\otimes 2} \vert$. Furthermore, let $C$ be canonically embedded in $\mathbf P^3$, then $3t$ is the complete intersection of two quadrics and a cubic surface.
We show that the GIT-quotient $\mathcal Q$ of the family of these complete intersections is rational and that there is a natural birational map between $\mathcal Q$ and $\mathcal S^{1/3}_4$.
\par
Adding up this result to the known picture we obtain the list of cases of genus $g \leq 4$ where the rationality of $\mathcal S^{1 / r}_g$ is proven. Here is the complementary list of unknown cases for $g \leq 4$:
\begin{itemize} \it
\item[$\circ$] Moduli of $4$-spin curves of genus $3$.
\item[$\circ$] Moduli of odd spin curves of genus $4$.
\item[$\circ$] Moduli of $6$-spin curves of genus $4$.
\end{itemize}
In particular it seems that the case of odd spin curves of genus 4 was not considered in the literature. Notice also that $\mathcal S^{1 / (2g-2)}_g$ splits into the union of two components: the moduli of pairs
$(C, \eta)$ such that $\eta^{\otimes g-1}$ is an even theta characteristic and the complementary component. We will denote them respectively by $$ \mathcal S^{1 / (2g-2) +}_g \ , \ \mathcal S^{1 / (2g-2) -}_g. $$
We will say that $(C, \eta)$ is an even (odd) $r$-spin curve if $\eta^{\otimes r}$ is an even (odd) theta characteristic. 
In the final part of this paper we almost complete the picture of the known results  in genus $g \leq 4$.  Building on quite related geometric constructions and methods, we prove the following theorems.
\begin{theorem} The moduli space of odd spin curves of genus 4 is rational. \end{theorem}
\begin{theorem} The moduli spaces of $4$-spin curves of genus 3 are rational. \end{theorem}
We have not found evidence to the uniruledness of $\mathcal S^{1 / r}_g$ in the only two missing cases in genus $g \leq 4$, namely for $\mathcal S^{1 / 6 +}_4$ and
$\mathcal S^{ 1 / 6 -}_4$. The same lack of evidence appears for further very low values of $g$, say $g \leq 7$ and $r \geq 3$.  Already for these cases, it could be interesting to apply some recent results on the structure of the Picard group of the Deligne-Mumford compactification of $\mathcal S^{1 / r}_g$ to obtain informations on the Kodaira dimension of these spaces, (cfr. for instance  \cite{pernigotti2013rational} and \cite{randal2012picard}). 
\medskip \par
\underline {Some frequently used notation and conventions:} \par
$\circ$ Let $C$ be a stable curve of genus $g$, we denote by $[C]$ its moduli point in $\overline {\mathcal M}_g$. In the same way we denote by $[C; L_1, \dots, L_m]$ the moduli
point of $(C;L_1, \dots, L_m)$, where $L_1 \dots L_m$ are line budles on $C$ of fixed degrees. \par
$\circ$ If $L$ is a line bundle on $X$ then $\vert L \vert$ denotes the linear system of the divisors of $X$ defined by the global sections of $L$. \par
$\circ $ Throughout the paper an elliptic curve $E$ is a smooth, connected curve $E$ of genus $1$, marked by one point $o$.
 
\section{3-spin curves of genus 4} 
Let $(C, \eta)$ be a spin curve of genus $g$ and order $r$. We will assume that $C$ is canonically embedded in $\mathbf P^{g-1}$. \par Putting $k =  [\frac {g-1}{\deg  \eta}] + 1$, we have $h^0(\eta^{\otimes k}) \geq 1$ by Riemann-Roch.
This implies that each effective divisor $t \in \vert \eta^{\otimes k} \vert$ satisfies the condition $rt = C \cdot F$, where $F$ is a hypersurface of degree $k$.
If $\deg \eta \ \text {divides $g$}$ then $\deg  t = g$ and we expect that $t$ is isolated, which is equivalent to $h^1(\eta^{\otimes k}) = 0$.  \par Let us focus  on the case $g=4$ and $r = 3$. 
In this situation $C\subset \mathbf P^3$ is a genus $4$ curve of degree $6$ and $t$ is a divisor in the linear system $\left|\eta^{\otimes 2}\right|$. Then $3t$  is a bicanonical divisor and there exists  a quadric surface 
$S$ such that $$ 3t=C \cdot S. $$  \begin{lemma} Let $C$ be a general curve of genus $4$, then $h^0(\eta) = 0$ for every $3$-spin curve $(C, \eta)$.  \end{lemma}
\begin{proof}   We can assume that $C = Q  \cap F$, where $Q$ is a fixed, smooth quadric and $F$ a cubic surface. Now assume $h^0(\eta) = 1$ for some cubic root  $\eta$ of $\omega_C$. 
Then there exist points $x, y \in C$ such that $x + y \in \vert \eta \vert$ and $3x + 3y = C \cdot H$, where $H \in \vert \mathcal O_Q(1) \vert$.  Let $\mathcal F$ be the family of complete intersections
$3x' + 3y' = C' \cdot H'$, where $H' \in \vert \mathcal O_Q(1) \vert$ and $C' \in \vert \mathcal O_Q(3) \vert$ is smooth. $\text{It is easy to see}$ that the action of $\Aut  Q$ on $\mathcal F$ has finitely many orbits. 
On the other hand, since $3x + 3y$ is a complete intersection, $\text {it follows}$ $\dim  \vert \mathcal I_{3x + 3y}(C) \vert = 8$, where $\mathcal I_{3x+3y}$ is the ideal sheaf of $3x + 3y$. 
But then, since the moduli space of $C$ is 9-dimensional, $C$ is not general: a contradiction. \end{proof}
From now on our spin curve $(C, \eta)$ will be sufficiently general. In particular we fix the following assumptions:
\begin{assumption} \ \par
\begin{itemize} \it
\item[$\circ$] $C$ is a complete intersection in $\mathbf P^3$ of a smooth quadric $Q$ and a cubic $F$,
\item[$\circ$] for each $x \in C$ one has $h^0(\mathcal O_C(3x)) = 1$,
\item[$\circ$] $h^0(\eta) = 0$ so that $h^0(\eta^{\otimes 2}) = 1$. 
\end{itemize} \end{assumption}
The second condition is just equivalent to say that the two $g^1_3$'s on $C$ have simple ramification. The third one is satisfied iff the unique effective divisor $t \in \vert \eta^{\otimes 2} \vert$ 
is not contained in any plane. \par It is clear that the locus of moduli of pairs $(C, \eta)$ satisfying these assumptions is a dense open subset of $\mathcal S^{1 / 3}_4$. 
It is also clear from the previous remarks that the bicanonical divisor $3t$ is a complete intersection scheme in the ambient space $\mathbf P^3$, namely 
$$ 3t = F \cdot Q \cdot S, $$
where $S$ is a quadric. This defines a second curve, we denote from now on as
$$
E := Q \cdot S.
$$
We point out that \it $E$ is uniquely defined by $(C, \eta)$. \rm $E$ is a quartic curve of arithmetic genus one. We will denote by $\mathcal I_{at}$ \it the ideal sheaf in $Q$  \rm of the divisor $at \subset C$.  Let $o \in t$ be a closed point,  we can fix local parameters $x, y$ at $o$ so that $y$ is a local equation of 
$C$ and $x$ restricts to a local parameter in $\mathcal O_{C,o}$. Then $\mathcal I_{at}$ is generated at $o$ by $x^{am}$ and $y$, where $m$ is the multiplicity of $t$ at $o$.  We observe that $3t$ is a  0-dimensional scheme of length 12, embedded in the smooth curve $C$. \par
Now assume for simplicity that $E$ is smooth. Since $3t = F \cdot Q \cdot S =  F \cdot E$, it follows that $3t$, as a divisor of $E$, belongs to $\vert \mathcal O_E(3) \vert$. Let us define
$$
\epsilon(1) := \mathcal O_E(t).
$$
Since we are assuming that $h^0(\mathcal O_C(t)) = 1$, we know that then $t$ is not contained in a plane. Hence the line bundle $\epsilon$
 is non trivial. On the other hand we have $3t \in \vert \mathcal O_E(3) \vert$ so that $\epsilon^{\otimes 3} \cong \mathcal O_E$. It follows that
\begin{lemma} $\epsilon$ is a non trivial 3-torsion element of $\rm{Pic}^0 E$. \end{lemma} 
Actually the condition that $E$ be smooth is satisfied as soon as the the pair $(C, \eta)$
is sufficently general. This is proven in the next theorem, where some useful conditions, satisfied by a general pair $(C, \eta)$, are summarized.
 \begin{theorem} On a dense open set $U \subset \mathcal S^{1 / 3}_4$ every point is the moduli point of a spin curve $(C, \eta)$ such that: 
\begin{enumerate} \it
\item[$1$] $(C, \eta)$ is general as in assumption 2.1,
\item[$2$] $E$ is a smooth quartic elliptic curve, 
\item [$3$] $t$ is a smooth divisor of $E$,
\item[$4$] $t \in \vert \epsilon(1) \vert $, where $\epsilon$ is a non trivial third root of $\mathcal O_E$.
\end{enumerate}
\end{theorem} 
\begin{proof}  We use the irreducibility of $\mathcal S^{1 / r}_g$ when $r$ is odd and $g \geq 2$, \cite{jarvis2000geometry}. $\mathcal S^{ 1 / 3}_4$ is irreducible,
so that every non empty open subset of it is dense. Conditions 1) and 2), 3), 4) are open on  families of triples $(C, \eta, E)$ hence they define open subsets of 
$\mathcal S^{1 / 3}_4$. We already know that the open set defined by 1) is not empty. Therefore, to prove the theorem, it suffices to produce one pair $(C, \eta)$ satisfying
2), 3), 4).   We start from a smooth elliptic quartic $E$. We have $E = Q \cdot S \subset \mathbf P^3$, where $Q, S$ are smooth quadrics.   Let $\epsilon \in \rm{Pic}^0 E$ be a non trivial element such that $\epsilon^{\otimes 3} \cong \mathcal O_E$. Since $\epsilon(1)$ is very ample, a general
$t \in \vert \epsilon(1) \vert$ is smooth and not contained in a plane. Note that $3t \in \vert \mathcal O_E(3) \vert$. Then, since $E$ is projectively normal, there exists a cubic surface $F$ such that
$$
3t = Q \cdot S \cdot F
$$
in the ambient space $\mathbf P^3$. Let $\mathcal I_{3t}$ be the ideal sheaf of $3t$ in $Q$, then we have $h^0(\mathcal I_{3t}(3)) = 5$. Moreover the  base locus of $\vert \mathcal I_{3t}(3) \vert$ is $3t$. Hence, by Bertini theorem, a general $C \in \vert \mathcal I_{3t}(3) \vert$ is smooth along $C - t$. To prove that a general $C$ is smooth along $t$ it suffices to produce one element with this property. This is the case for $E + L$, where $L$ is a general plane section. Let $C \in \vert \mathcal I_{3t}(3) \vert$ be smooth and let $\eta := \mathcal O_C(1 -t)$. $(C, \eta)$ is a spin curve of order 3
satisfying 2), 3), 4).
\end{proof}

\section{Projective bundles related to \texorpdfstring{$\mathcal S^{1 / 3}_4$}{moduli of 3-roots of genus 4}}
 
 Let $(C, \eta)$ be a general spin curve of order 3 and genus 4. We keep the previous conventions, so that $C$ is canonically embedded in $\mathbf P^3$ as $Q \cap F$.  \par It follows from the above theorem that 
 the moduli point $[C, \eta]$ uniquely defines, up to isomorphisms, a triple $(E, \epsilon, t)$ such that $E$ is a smooth quartic elliptic curve in $\mathbf P^3$ and $\epsilon$ is
a non trivial third root of $\mathcal O_E$. \par Moreover $t$ is a  smooth element of $\vert \epsilon(1) \vert$ and $3t$ is a complete intersection $$ 3t = C \cdot E = F \cdot Q \cdot S \subset \mathbf P^3,$$ where $S$ is a quadric. As a divisor in $C$, $t$ is the the unique element of $\vert \eta^{\otimes 2} \vert$. In order to prove the rationality of $\mathcal S^{1 / 3}_4$ our strategy is as follows. We consider the moduli space of elliptic curves $E$ endowed with a non trivial 3-torsion element of $\rm{Pic}^0 E$, namely
$$
\mathcal{R}_{1,3}:=\{[E,\varepsilon] \ \mid \ g(E)=1, \quad \varepsilon\neq \mathcal{O}_E, \quad \varepsilon^{\otimes 3}\cong \mathcal{O}_E\}.
$$
Over it we have the moduli space $\mathcal P_{1,4}$  of triples $(E, \epsilon, H)$ such that $H \in \Pic^4 E$. This can be also defined via the Cartesian square
$$
\begin{CD}
{\mathcal P_{4,1}} @>>> {\mathcal P ic_{4,1}} \\
@VVV @VVV \\
{\mathcal R_{1,3}} @>>> {\mathcal M_1.} \\
\end{CD}
$$
As usual, $\mathcal P ic_{4, 1}$ denotes the universal Picard variety, that is, the moduli space of pairs $(H, E)$ such that $E$ is an elliptic curve and $H \in Pic^4 \ E$. \par
The space $\mathcal P_{4, 1}$ is a rational surface.
Proving its unirationality, so that the rationality follows,  is easy. Starting from $\mathcal P_{4,1}$ we construct a suitable ``tower" 
$$
\mathbb P_c \stackrel {c} \to \mathbb P_b \stackrel {b} \to \mathbb P_a \stackrel{a} \to  \mathcal P_{4,1}
$$
of projective bundles $a, b, c$. Clearly, as a ``tower" of projective bundles over a rational base, $\mathbb P$ is rational. Let $\phi: \mathcal S^{1 / 3}_4 \to \mathcal P_{4 , 1}$ be the rational map defined as 
follows: $\phi ([C, \eta]) := [E, \epsilon]$. Then we will show 
that $\phi$ factors through a natural birational map between $\mathcal S^{1 / 3}_4$ and $\mathbb P_c$, so proving that $\mathcal S^{1 / 3}_4$ is rational. In the next subsections we produce the 
projective bundles which are needed.

\subsection{The ambient bundle \texorpdfstring{$\mathbb P$}{}}
Let us start with the universal elliptic curve over $\mathcal{M}_{1}$ and its pull-back $\mathcal{E} \to \mathcal{R}_{1,3}$. As is well known there exists a Poincar\'e bundle $\mathcal P$
on the fibre  product $\mathcal{P}_{4,1}\times_U \mathcal {E}$, where $U \subset \mathcal R_{1,3}$ is a suitable dense open set. In particular the restriction of $\mathcal P$ to the fibre at $[E, \epsilon, H]$ 
of the projection map $$ \alpha: \mathcal P_{4,1} \times_U \times \mathcal E \to \mathcal P_{4,1}$$  is given by
$
\mathcal{P}\otimes\mathcal{O}_{\{[E,\varepsilon,H]\}\times E}\cong H.
$ Note that $\left(\alpha_*\mathcal{P}\right)_{[E,\varepsilon,H]}=H^0(H)$ has constant dimension 4.  Let $\mathcal H := \alpha_*\mathcal{P}$; then, by Grauert's theorem, $\mathcal H$  is a vector bundle of rank $4$ 
over $\mathcal{P}_{4,1}$. We define the the ambient bundle $\mathbb P$ as follows: 
$$
\mathbb P := \mathbf P \mathcal H^*.
$$
Its structure map will be denoted as $p: \mathbb P \to \mathcal P_{4,1}$. It  is a $\mathbf P^3$-bundle over $\mathcal{P}_{4,1}$. In particular the tautological bundle $\mathcal O_{\mathbb P}(1)$ defines an embedding
$$
\mathcal P_{4,1} \times_U \mathcal E \subset \mathbb P.
$$
At $x :=[E, \epsilon, H]$ this is the embedding $E \subset  \mathbb P_x = \mathbf PH^0(H)^*$ defined by $H$.

\subsection{The bundle of quadrics \texorpdfstring{$a: \mathbb P_a \to \mathcal P_{4,1}$}{}}
Let us consider the map 
$$ \mu: \mathrm{Sym}^2\mathcal{H}\to \alpha_*(\mathcal{P}^{\otimes 2}) $$
of vector bundles on $\mathcal P_{4,1}$. At $x := [E, \epsilon, H]$ we have $\alpha_*(\mathcal P^{\otimes 2})_x = H^0(H^{\otimes 2})$ and
$$
\mu_x: \mathrm{Sym}^2 H^0(H) \to H^0(H^{\otimes 2})
$$
is the multiplication map. Putting  $\mathcal{Q}:=\ker \mu$ and $\mathbb P_a :=\mathbf P\mathcal{Q}$, we denote as
$$
a: \mathbb P_a \to \mathcal P_{4,1}
$$
the structure map. The bundle $a$ is a $\mathbf P^1$-bundle and the fibre $\mathbb P_x$ parametrizes the quadrics containing the tautological embedding $E \subset \mathbb P_x$ defined by $H$.

\subsection{The \texorpdfstring{$\mathbf P^3$-bundle $b: \mathbb P_b \to \mathbb P_a$}{first P3 bundle}}
At first we define the $\mathbf P^3$-bundle
$$
e: \mathbb P_e \to \mathcal P_{4,1}.
$$
Its fibre $\mathbb P_{e,x}$ will be $\vert \varepsilon\otimes H \vert$ at $x := [E, \epsilon,H]$. On $\mathcal P_{4,1} \times_U \mathcal E$ we fix a vector bundle
$\mathcal N$ whose restriction to the fibre of $\alpha: \mathcal P_{4,1} \times_U \mathcal E \to \mathcal P_{4,1}$ at $x$ is
$$
\mathcal N \otimes \mathcal O_{\alpha^*x} \cong \epsilon.
$$
The construction of $\mathcal N$ is standard: let $\beta: \mathcal P_{4,1} \times_U \mathcal E \to \mathcal R_{1,3} \times_U \mathcal E$ be the natural map. Then 
we define $\mathcal N := \beta^* \mathcal L$, where $\mathcal L$ is a Poincar\'e bundle on $\mathcal R_{1,3} \times_U \mathcal E$. Note that $\mathcal L$ restricted to the fibre at $[E, \epsilon]$ of the projection $\gamma: \mathcal R_{1,3} \times_U \mathcal E \to \mathcal R_{1,3}$ is the
line bundle $\epsilon$. We consider the tensor product $\mathcal H \otimes \mathcal N$ and finally $\alpha_*( \mathcal H \otimes \mathcal N)$. The latter is
a rank 4 vector bundle with fibre $H^0(H \otimes \epsilon)$ at $x$. We define
$$
\mathbb P_b := a^* \mathbf P \alpha_*(\mathcal H \otimes \epsilon).
$$
$\mathbb P_b$ is a $\mathbf P^3$-bundle over $\mathbb P_a$. The fibre at $x$ of the map $a \circ b: \mathbb P_b \to \mathcal P_{4,1}$  is the Segre product
$\vert \epsilon \otimes H \vert \times \vert \mathcal I_E(2) \vert$, where $\mathcal I_E$ is the ideal sheaf of the embedding $E \subset \mathbb P_x$. \par

\subsection{The \texorpdfstring{$\mathbf P^3$-bundle $c: \mathbb P_c \to \mathbb P_b$}{second P3 bundle}}
In the fibre product $\mathbb P_b \times_{\mathcal P_{4,1} } \mathbb P$ we define the following subvarieties
$$
\bf  t \subset \bf E \subset \bf Q \subset \mathbb P_b \times_{\mathcal P_{4,1}} \mathbb P.
$$
Let $o \in \mathbb P_b \times_U \times \mathbb P$, then $o$ defines a pair $(v,z)$ where $z \in \mathbb P_x$ and $x := a \circ b (o)$ 
$= [E, \epsilon, H]$. Moreover, the point $o$ is an element $t \in \vert \epsilon \otimes H \vert$ of the fibre of $\mathbb P_b$ at $b(o)$. Finally $b(o)$ is an 
element $Q \in \vert \mathcal I_E(2) \vert$, where $\mathcal I_{\bf E}$ is the ideal sheaf of the tautological embedding $E \subset \mathbb P_x$. Clearly we
have $\bf t \subset E \subset Q$. \par The conditions $z \in{\bf t}$, $z \in {\bf E}$, $z \in {\bf Q}$ respectively define the closed sets $\bf t$, $\bf E$, $ \bf Q$.
In particular $\bf E$ is a natural embedding of $\mathcal P_{1,4} \times_U \mathcal E$ in  $\mathbb P_b \times_{\mathcal P_{4,1}} \mathbb P$ and $ \bf t$ is a Weil divisor in
$\bf E$. Let us consider the standard exact sequence
\[
0 \to \mathcal I_{3\bf t}\to \mathcal O_{\bf Q} \to \mathcal O_{3{\bf t}}\to 0
\]
where $\mathcal I_{3\bf t}$ is the ideal sheaf of $\bf t$ in $\mathbf Q$. We pull-back the line bundle
$\mathcal O_{\mathbb P}(3)$ to the fibre product $\mathbb P_b \times_{\mathcal P_{4, 1}} \mathbb P$ and tensor the above exact sequence by it. The resulting exact sequence
is denoted in the following way:
\[
0 \to \mathcal I_{3\bf t}(3)\to \mathcal{O}_{\bf Q}(3)\to\mathcal{O}_{3\bf t}(3)\to 0.
\]
Let  $\beta: \mathbb P_b \otimes \mathbb P \to \mathbb P_b$ be the projection onto $\mathbb P_b$. Then we apply the push-down functor $\beta_*$ to this new exact sequence. We obtain the exact sequence
\[
0 \to \beta_*\mathcal I_{3\bf t}(3) \to \beta_* \mathcal O_{\bf Q}(3)\to\beta_{*|_{\bf Q}}\mathcal{O}_{3{\bf t}}(3) \to R^1\beta_* \mathcal I_{3\bf t}(3) = 0.
\]
Here the sheaf $R^1 \beta_* \mathcal I_{3 \bf t}(3)$ is zero because at any point $p = (t, Q, [E, \epsilon, H]) \in \mathbb P_b$ its fibre  is $H^1(\mathcal I_{3t / Q}(3)) = 0$. Notice also that the sheaf
$\mathcal F := \beta_* \mathcal I_{3 \bf t}(3)$ is a rank 5 vector bundle with fibre $H^0(\mathcal I_{3t / Q}(3))$ at the same point $p$. Finally we define
$$
\mathbb P_c := \mathbf P \mathcal F.
$$
We denote the structure map of this $\mathbf P^4$-bundle as $c: \mathbb P_c \to \mathbb P_b$.  The fibre of $c$ at $p$ is the linear system of cubic sections $C$ of $Q$ containing the scheme
$3t \subset E$. Notice that a smooth $C$ is a canonical curve of genus 4 endowed with the order 3 spin structure $$ \eta := \omega_C(- t). $$

\section{The rationality of \texorpdfstring{$\mathcal S^{1 / 3}_4$}{the moduli space of 3-roots of genus 4}} 

Let $\mathcal I_{2t / \mathbf P^3}$ be the ideal sheaf of $2t \subset C \subset \mathbf P^3$. Notice also that
\begin{lemma} $\vert \mathcal I_{2t / \mathbf P^3}(2) \vert$ is a pencil of quadrics with base locus $E$. \end{lemma}
\begin{proof} Observe that $\omega_C^{\otimes 2}(-2t) \cong \eta^{\otimes 2}$. Moreover, this is also the sheaf $\mathcal I_{2t / C}(2)$. Consider the standard exact sequence of ideal sheaves
$$
0 \to \mathcal I_{C / \mathbf P^3}(2) \to \mathcal I_{2t / \mathbf P^3}(2) \to \eta^{\otimes 2} \to 0.
$$
Since we have $h^0(\mathcal I_{ C / \mathbf P^3}(2)) = h^0(\eta^{\otimes 2}) = 1$, the statement follows.  \end{proof}

Due to the latter construction there exists a natural moduli map
$$
\phi: \mathbb P_c \to \mathcal S^{1 / 3}_4
$$
which sends a point $z = (C, t, Q, [E, \epsilon, H]) \in \mathbb P_c$ to the point $$ \phi(z) := (C, \eta), $$ with $\eta = \omega_C(-t)$. Clearly $\phi$ is defined at $z$ iff $C$ is smooth.
 Since $\mathbb P_c$ is rational we can finally deduce the rationality of $\mathcal S^{1 / 3}_4$, stated in the Introduction. We show that
 \begin{theorem} The map $\phi: \mathbb P_c \to \mathcal S^{ 1 / 3}_4$ is birational, so that  $\mathcal{S}^{1/3}_4$ is rational.
\end{theorem}
\begin{proof} At first we show that the map $\phi$ is dominant. Starting with a general point $[C,\eta] \in \mathcal{S}^{1/3}_4$ it is possible to reconstruct a point $z = (C, t, Q, [E, \epsilon, H]) \in \mathbb P_c$
such that $\phi(z) = [C, \eta]$. Indeed $t$ is the unique element of $ \vert \eta^{\otimes 2} \vert$. Then, from the canonical embedding $C \subset \mathbf P^3$, we reconstruct $E$ as the smooth base locus of the pencil of
quadrics $\vert \mathcal I_{2t}(2) \vert$, considered above. Then we have $H := \mathcal O_E(1)$ and $\epsilon := H(-t)$. The quadric $Q$ is the unique quadric of $\vert \mathcal I_{2t/ \mathbf P^3}(2) \vert$ containing
$C$. It is clear that $[C, \eta] = \phi(z)$, with $z = (C, t, Q, [E, \epsilon, H])$. Conversely  the inverse map of $\phi$ is well-defined too. Starting from a general $[C, \eta]$ the point $z$ is indeed uniquely 
reconstructed as above. Hence $\phi^{-1}$ is well defined and $\phi$ is birational. 
\end{proof} 
In the next sections we prove the other rationality results announced in the Introduction.

\section{The rationality of \texorpdfstring{$\mathcal S^-_4$}{the moduli space of odd theta of genus 4}}

We start from an odd spin curve $(C, \eta)$ of genus 4. As in the previous sections, $C$ will be sufficiently general. Thus, passing to its canonical model, we have
$$
C \subset  Q \subset \mathbf P^3,
$$
where $Q = \mathbf P^1 \times \mathbf P^1$ is a smooth quadric and $C$ has bidegree $(3,3)$  in it. Since $\eta$ is odd, there exists a unique $d \in \vert \eta \vert$ and we have
$$
2d = L \cdot C,
$$
where $L$ is a plane section of $Q$ and a conic tritangent to $C$. The condition that both $d$ and $L$ be smooth clearly defines an open set $U \subset  \mathcal S^-_4$. Furthermore it is easily seen that $U \neq \emptyset$. Then, since $\mathcal S^-_4$ is irreducible, the next lemma follows.
 \begin{lemma} 
For a general $C$  both the divisor $d$ and the conic $L$ are smooth. 
\end{lemma}
Let $o_1, o_2, o_3$ be the three points of $d$. They are not collinear because $h^0(\eta) = 1$. Hence we can fix projective coordinates $(x_0:x_1) \times (y_0:y_1)$ on $\mathbf P^1 \times \mathbf P^1$ so that
$$
o_1 = (1:0) \times (1:0) \ , \ o_2 = (0:1) \times (0:1) \ , \ o_3 = (1:1) \times (1:1).
$$
In particular we can assume that these points are in the diagonal 
$$ 
L := \lbrace x_0y_1 - x_1y_0 = 0 \rbrace
$$
of $\mathbf P^1 \times \mathbf P^1$. Let $\mathcal I_{2d}$ be the ideal sheaf of $2d$ in $\mathbf P^1 \times \mathbf P^1$ and let
$$
I := H^0(\mathcal I_{2d}(3,3)).
$$
We consider the 9-dimensional linear system $ \mathbf P I$. This is endowed with the map
$$
m: \mathbf PI \to \mathcal S^-_4
$$
defined as follows. Let $C \in \mathbf PI$ be smooth, then $m(C) := [C, \eta]$, where $\eta := \mathcal O_C(o_1 + o_2 + o_3)$. It is clear from the construction that $m$ is dominant. Let
\begin{equation}\label{groupG}
\mathbb G \subset \Aut  \mathbf P^1 \times \mathbf P^1
\end{equation}
be the stabilizer of the set $\lbrace o_1 , o_2, o_3 \rbrace$. We have:
\begin{lemma} \label{lemma_gen4}
Assume $C_1, C_2 \in \mathbf PI$ are smooth. Then $m(C_1) = m(C_2)$ if and only if  $C_2 = \alpha(C_1)$ for some $\alpha \in \mathbb G$.
\end{lemma}
\begin{proof} Let $m(C_i) = [C_i, \eta_i]$, $i = 1,2$. If $m(C_1) = m(C_2)$ there exists a biregular map $a: C_2 \to C_1$. Since $\mathcal O_{C_i}(1,1) \cong \omega_{C_i}$, it follows that $a$
induces an isomorphism $a^*: H^0(\mathcal O_{C_1}(1,1)) \to H^0(\mathcal O_{C_2}(1,1))$. This implies that $a$ is induced by some $\alpha \in \Aut  \mathbf P^1 \times \mathbf P^1$.
Furthermore, the
condition $m(C_1) = m(C_2)$ also implies that $a^* \mathcal O_{C_2}(o_1+o_2+o_3) \cong \mathcal O_{C_1}(o_1+o_2+o_3)$.  Hence $\alpha \in \mathbb G$. The converse is obvious. 
\end{proof}
Now observe that $\mathbb G$ acts, in the natural way, on $\mathbf PI$ and that $m: \mathbf PI \to \mathcal S^-_4$ is dominant. Then, as an immediate consequence of the previous lemma, we have
\begin{corollary} 
 $\mathcal S^-_4$ is birational  to the quotient $\mathbf PI / \mathbb G$. 
\end{corollary}
Thus the rationality of $\mathcal S^-_4$  follows if $\mathbf PI / \mathbb G$ is rational. In order to prove this, we preliminarily describe the group $\mathbb G$ and its action on $\mathbf PI$.  
We recall that the natural inclusion $\Aut  \mathbf P^1 \times \Aut  \mathbf P^1 \subset \Aut  \mathbf P^1 \times  \mathbf P^1$ induces the exact sequence
$$
0 \to \Aut  \mathbf P^1 \times \Aut  \mathbf P^1 \to \Aut  \mathbf P^1 \times \mathbf P^1 \to \mathbb Z_2 \to 0,
$$
where $\mathbb Z_2$ is generated by the class of the projective involution 
$$
\iota: \mathbf P^1 \times \mathbf P^1 \to \mathbf P^1 \times \mathbf P^1
$$ 
exchanging the factors. From the above exact sequence we have the exact sequence
$$
0 \to \mathbb G_3 \to \mathbb G \to \mathbb Z_2 \to 0.
$$
Here $\mathbb G_3 $ denotes the stabilizer of the set $O := \lbrace o_1 , o_2 , o_3 \rbrace$ in $\Aut  \mathbf P^1 \times \Aut \mathbf P^1$. Since $O$ is a subset of the diagonal $L$, $L$ itself is fixed by $\mathbb G_3$. In particular it follows that $\mathbb G_3$ is the diagonal embedding in $\Aut  \mathbf P^1 \times \Aut  \mathbf P^1$ of the stabilizer of $\lbrace o_1 , o_2 , o_3 \rbrace$ in $\Aut  L$. 
As is very well known, this is a copy of the symmetric group $S_3$. \par 
Now we proceed to an elementary and explicit description of the $\mathbb G$-invariant subspaces of $\mathbf PI$. From it the rationality of $\mathbf PI / \mathbb G$ will follow.
We fix the notation $l $ $:= x_0y_1 - x_1y_0$ for  the equation of the diagonal $L$. Let $$ R = \oplus_{a,b \in \mathbb Z} R_{a,b}$$ be the coordinate ring of $\mathbf P^1 \times \mathbf P^1$, where $R_{a,b}$ is the vector space of forms of bidegree $a, b$. 
We can assume that $\iota^*: R \to R$ is the involution such that $\iota^*x_i = y_i$, $i = 0, 1$. On the other hand let 
$$ 
h_1 := x_0(y_1-y_0) + y_0(x_1-x_0) \ , \ h_2 := x_1(y_0-y_1) + y_1(x_0-x_1) \ , \ h_3 := x_0y_1 + x_1y_0,
$$
so that $\lbrace l, h_1, h_2, h_3 \rbrace$ is a basis of $R_{1,1}$. We can also assume that, for each $\sigma \in \mathbb G_3$, the map $\sigma^*: R \to R$ is such that $\sigma^* l = l$ and $\sigma^*$ permutes the elements
of the set $\lbrace h_1, h_2, h_3 \rbrace$. Then we observe that the eigenspaces of $\iota^*: R_{1, 1} \to R_{1,1}$ are
$$
R_{1,1}^- = < l > \ , \ R_{1,1}^+ = < h_1 , h_2 , h_3 >.
$$ 
This implies that
$$
R_{1,1} = < l > \oplus <h_1 + h_2 + h_3 > \oplus <h_1 - h_3 , h_2 - h_3>
$$
where all the summands are $\mathbb G$-invariant. Considering the multiplication map
$$
\mu: \mathrm{Sym}^2 R_{1,1} \to R_{2,2}
$$
one can check that
$$
\Ker \mu = < h_3^2 - l^2 - (h_1-h_3)(h_2-h_3)>.
$$
Then, putting $h := h_1 + h_2 + h_3$ and $h_{ij}�:= h_i - h_j$, it is easy to deduce that the eigenspaces of $\iota^*: R_{2,2} \to R_{2,2}$ decompose as follows:
$$
R_{2,2}^+ = <h^2> \oplus < hh_{13} \ , \ hh_{23} > \oplus < h^2_{13} \ , \ h^2_{23}> \oplus <h_{13}h_{23} >
$$
and
$$
R_{2,2}^- = < lh > \oplus < lh_{13} \ , \ lh_{23} >,
$$
where each summand appearing above is $\mathbb G$-invariant. Finally, we consider the vector space $I$ and observe that, taking the multiplication by $l$, we
have an injection $$ < l > \otimes R_{2,2} \hookrightarrow I.$$ Its image $l R_{2,2}\subset I$ is a subspace codimension one. Moreover we have 
$$
l R^+_{2,2} \subseteq I^- \ , \ lR^-_{2,2} \subseteq I^+,
$$
where $I^+ , I^-$ are the eigenspaces of $\iota^*: I \to I$. Let us consider 
$$c = x_0x_1(x_0-x_1)+y_0y_1(y_0-y_1).$$
Notice that $c \in I$ and that ${\rm div }(c)$ is $\mathbb G$-invariant. Indeed, 
${\rm div} ( c)$ is the union of the six lines in the quadric $Q = \mathbf P^1 \times \mathbf P^1$ passing through the points $o_1, o_2, o_3$. Notice also that $c$
is not in $lR_{2,2}$, in particular $I = <c> \oplus  lR_{2,2}$. Notice also that $\iota^*c = c$. \par Summing all the previous remarks up, we can finally describe the eigenspaces of $\iota^*: I \to I$ 
and their decompositions as a direct sum of $\mathbb G$-invariant summands.  
\begin{lemma} Let $I^+ , I^-$ be the eigenspaces of $\iota^*: I \to I$, then we have
\begin{itemize} \item[$\circ$]
$I^+  =  < c > \oplus  < l^2h > \oplus < l^2h_{13} \ , \ l^2h_{23} >$.
\item[$\circ$] $I^- =  <lh^2> \oplus < lhh_{13} \ , \ lhh_{23} > \oplus < lh^2_{13} \ , \ lh^2_{23}> \oplus <lh_{13}h_{23} >$,
\end{itemize}
where each summand is an irreducible representation of $\mathbb G$.
\end{lemma}
Now it is straightforward to conclude. For instance let us consider $$ B := \mathbf PI^+ \times \mathbf PI^- $$ and then the variety
$$
\mathbb P := \lbrace (x, p) \in \mathbf PI \times B \ \mid \ x \in \mathbb P_p \rbrace \subset \mathbf PI \times B,
$$
where $p := (p^+, p^-) \in \mathbf P I^+ \times \mathbf P I^-$ and $\mathbb P_p$ denotes the line joining $p^+$ and $p^-$. The variety $\mathbb P$ is endowed with its two natural projections 
$$
\begin{CD}
{\mathbf PI} @<{\beta}<< {\mathbb P} @>{\alpha}>> B. \\
\end{CD}
$$
Note that $\beta: \mathbb P \to \mathbf PI$ is birational, since there exists a unique line $\mathbb P_p$ passing through a point in $\mathbf PI - (\mathbf PI^+ \cup \mathbf PI^-)$. Moreover
$$ \alpha: \mathbb P \to B $$ is a $\mathbf P^1$-bundle structure with fibre $\mathbb P_p$ at the point $p = (p^+, p^-) \in B$. It is also clear that the action of $\mathbb G$ on $\mathbf PI$ induces an action of $\mathbb G$ on $\mathbb P$ and that
$$
\mathbf PI/\mathbb G \cong \mathbb P / \mathbb G.
$$
More precisely, the map $\iota^*$ acts as the identity on $B$, since its two factors are projectivized eigenspaces of $\iota^*$.  Moreover each fibre $\mathbb P_p$ of $\alpha$ is $\iota^*$-invariant. Indeed $\iota^* / \mathbb P_p$ is a projective involution with fixed points $p^+$, $p^-$ on the line $\mathbb P_p$. \par Note that the induced action of $\mathbb G_3$ on $B$ is faithful, since the 2-dimensional summands of $I^{\pm}$ are standard representations
of $S_3$.  Furthermore $\mathbb G_3$ acts linearly on the fibres of $\alpha: \mathbb P \to B$. \par Indeed  consider any $\phi \in \mathbb G_3$ and any $p = (p^+, p^-) \in B$. Then $\phi( \mathbb P_p)$ is the line $\mathbb P_{\phi(p)}$, where $\phi(p) = (\phi(p^+), \phi(p^-))$. In particular the map $\phi / \mathbb P_p \to \mathbb P_{\phi(p)}$ is a projective isomorphism. Let
$$
\overline {\mathbb P} := \mathbb P / \mathbb G_3;
$$
the latter remarks imply that $\alpha: \mathbb P \to B$ descends to a $\mathbf P^1$-bundle 
$$
\overline {\alpha}: \overline {\mathbb P} \to B/\mathbb G_3,
$$
over a non empty open set $U \subset B / \mathbb G_3$. Now let us consider $\iota \in \mathbb G$ and the involution $\iota: \mathbb P \to \mathbb P$ due to the action of $\mathbb G$ on $\mathbb P$. It is clear from the previous construction that $\iota$ descends to an involution
$$
\overline {\iota}: \overline {\mathbb P} \to \overline {\mathbb P}
$$ 
which is fixing each fibre of $\overline {\alpha}$ and acts linearly on it. Passing to the quotient $$ \hat {\mathbb P} := \overline{\mathbb P} / < \overline {\iota} >, $$ it follows that
$\overline {\alpha}$ induces a $\mathbf P^1$-bundle structure $\hat  {\alpha}: \hat{\mathbb P} \to B / \mathbb G_3$. 
\begin{remark} \rm Actually $\hat{\alpha}$ has two natural sections $s^{\pm}: B / \mathbb G_3 \to \hat {\mathbb P}$. They are defined as follows: let $\overline p \in B / \mathbb G$ be the orbit of $p = (p^+, p^-) \in B$.  Then the fixed points of $\overline {\iota}: \overline {\mathbb P}_{\overline p} \to \overline {\mathbb P}_{\overline p}$ are the orbits
$\overline p^+, \overline p^-$ of $p^+, p^-$. Passing to the quotient by $\overline \iota$ they define two distinguished points $\hat p^+, \hat p^- \in \hat {\mathbb P}_{\overline p}$: by definition $\hat p^{\pm} = s^{\pm}(\overline p)$.  \end{remark} 
 \begin{theorem}
The quotient $\mathbb P / \mathbb G$ is rational.
\end{theorem}
\begin{proof} Since $\mathbb P / \mathbb G \cong \hat{\mathbb P}$ and $\hat \alpha: \hat {\mathbb P} \to B / \mathbb G_3$ is a $\mathbf P^1$-bundle, the preceeding remarks imply that $\mathbb P / \mathbb G \cong B / \mathbb G_3 \times \mathbf P^1$. Hence it remains to show the rationality of $B / \mathbb G_3$. This is now straightforward: we have
$B = \mathbf PI^+ \times \mathbf PI^-$ and  $\mathbb G_3$ acts linearly on both factors. Considering $B$ as the trivial projective bundle over $\mathbf PI^+$, it follows that $B / \mathbb G_3$ is a $\mathbf P^5$-bundle over $\mathbf PI^+ / \mathbb G_3$. The rationality of $\mathbf PI^+ / \mathbb G_3$ is a standard property. Since $\mathbf PI^+ = \mathbf P^3$, it is easily proven considering the decomposition of $I^+$  as a sum of irreducible representations of $\mathbb G_3$. Hence $B / \mathbb G_3$ is rational.  \end{proof}
We have already proved that $\mathcal S^-_4$ is birational to $\mathbb P / \mathbb G$. Hence it follows:
\begin{corollary}
The moduli space of odd spin curves of genus $4$ is rational. 
\end{corollary}
 \section{The rationality of \texorpdfstring{$\mathcal S^{1/4-}_3$}{the moduli space of odd 4-roots of genus 3}}
The rationality result to be proven in this section naturally relies on the geometry of odd spin curves of genus $4$ considered above. To see this relation let us fix from now on a
general curve $C$ of genus three and two distinct points $n_1, n_2 \in C$. As is well known, the line bundle $\omega_C(n_1+n_2)$ defines a morphism $\phi: C \to \mathbf P^3$ such that
$$ C_n := \phi(C) \subset Q \subset \mathbf P^3. $$
Here $Q := \mathbf P^1 \times \mathbf P^1$ is a smooth quadric and the unique quadric through $C_n$.   Moreover $C_n$ is a curve of bidegree $(3,3)$ in $\mathbf P^1 \times \mathbf P^1$ with exactly one node $n := \phi(n_1) = \phi(n_2)$, see \cite{green1986projective}. Let $R_1, R_2$ be the lines in $Q$ containing $n$. Then the pull-back by $\phi: C \to Q$ of the divisor $R_i$ is $$ n_1+n_2+m_i, \ i=1,2 $$ 
where $n_1+n_2+m_1+m_2$ is the canonical divisor of $C$ containing the points $n_1$ and $n_2$. Moreover $\vert n_1 + n_2 + m_i \vert$ is the pencil $\vert \omega_C(-m_j) \vert$, where $j \neq i$.
The condition that $\mathcal O_C(n_1+n_2)$ is a theta characteristic is reflected by the projective model $C_n$ as follows:
 \begin{lemma} 
Let $R_1$ and $R_2$ be the two lines of $Q$ passing through the node $n$. Then the following conditions are equivalent:
\begin{enumerate}
\item $R_1$, $R_2$ are tangent to the branches of $C_n$ at $n$.
\item $\mathcal O_C(n_1+n_2)$ is a theta characteristic.
\end{enumerate}
\end{lemma}
\begin{proof} (i) $\Rightarrow$ (ii): Since $R_1 + R_2 \in \vert \mathcal O_Q(1) \vert$ it follows $\phi^*(R_1 + R_2) = 3n_1 + 3n_2 \in \vert \omega_C(n_1 + n_2) \vert$.  Hence
$2n_1 + 2n_2$ is a canonical divisor and $\mathcal O_C(n_1 + n_2)$ is a theta characteristic.  (ii) $\Rightarrow$ (i):  Since $2n_1 + 2n_2$ is a canonical divisor then 
$ \vert 2n_i + n_j \vert$ is a pencil, $j \neq i$. Let $\phi_i: C \to \mathbf P^1$ be the map defined by $\vert 2n_i+n_j\vert$, then $\phi: C \to Q$ is the product map $\phi_1 \times \phi_2$. 
Therefore, up to reindexing, we have $2n_i + n_j = \phi^*R_i$. Hence $R_i$ is tangent to a branch of $n$.
   \end{proof}  Assume now that $[C, \eta]$ is a general point of $\mathcal S^{1 / 4 -}_3$ so that $\eta^{\otimes 2}$ is an odd theta characteristic on $C$. This is equivalent to say that there
exist two distinct points $n_1, n_2 \in C$ such that $$ \mathcal O_C(n_1 + n_2) \cong \eta^{\otimes 2} \ {\rm and } \ \eta^{\otimes 6} \cong \omega_C(n_1+n_2) \cong \mathcal O_C(3n_1+3n_2).$$
Considering the morphism $\phi$ defined by $\eta^{\otimes 6}$, we have as above that its image 
$$
C_n \subset Q \subset \mathbf P^3
$$
is a curve with exactly one node $n = \phi(n_1) = \phi(n_2)$ and no other singular point. Now we observe that the linear system $\vert \eta^{\otimes 6} \vert$ contains the two distinct elements:
\begin{itemize}
\item[$\circ$] $3n_1 + 3n_2$, where $n_1 + n_2 \in \vert \eta^{\otimes 2} \vert$,
\item[$\circ$] $2o_1 + 2o_2 + 2o_3$, where $o_1 + o_2 + o_3 \in \vert \eta^{\otimes 3} \vert$.
\end{itemize}
\begin{lemma} One has $h^0(\eta) = 0$, so that $h^0(\mathcal O_C(o_1+o_2+o_3)) = 1$.   \end{lemma}
\begin{proof} If $h^0(\eta) \geq 1$ then $\eta \cong \mathcal O_C(p)$ for some point $p \in C$. But then $4p \in \vert \omega_C \vert$, which is impossible on a general $C$ of genus $3$. Now observe that $\omega_C(-o_1-o_2-o_3) \cong \eta$. Since $h^0(\eta) = 0$ it follows $h^0(\mathcal O_C(o_1+o_2+o_3)) = 1$ by Riemann-Roch. \end{proof}
\begin{lemma} The points $o_1, o_2, o_3$ are distinct and $\lbrace o_1 \ o_2 \ o_3 \rbrace \cap \lbrace n_1 \ n_2 \rbrace = \emptyset $.
Moreover one has $2o_1 + 2o_2 + 2o_3 = L \cdot C_n$ where $L \in \vert \mathcal O_Q(1) \vert$ is  smooth.
\end{lemma}
\begin{proof} It suffices to produce one pair $(C, \eta)$ satisfying the statement. Fix in $\mathbf P^2$ five general points $o_1, o_2, o_3, n_1, n_2$ and let $L$ be the conic through them. Consider the linear system $\Sigma$ of all quartics $C$ which are tangent to $L$ at $o_1, o_2, o_3$ and tangent to the line $<n_1 ,  n_2>$ at $n_1, n_2$. It is easy to check that the general $C \in \Sigma$ is smooth. Let $\eta = \mathcal O_C(o_1 + o_2 + o_3 - n_1 - n_2)$, then $(C, \eta)$ satisfies the statement. \end{proof}
\begin{remark} \rm As above let $\mathcal O_C(n_1 + n_2)$ be an odd theta characteristic and let $C_n \subset \mathbf P^3$ be the image of the map defined by $\vert \omega_C(n_1+n_2) \vert$. It follows from the previous discussion that there exists a bijection between the set of square roots $\eta$ of $\mathcal O_C(n_1+n_2)$ and the set of tritangent planes $P$ to $C_n - \lbrace n \rbrace$. This bijection associates to $P$ the line bundle $\eta = \mathcal O_C(o_1 + o_2 + o_3 - n_1 - n_2)$, where $P \cdot C_n = 2o_1 + 2o_2 + 2o_3$.
\end{remark}
 To prove the rationality result of this section we  proceed as in the previous one. We fix coordinates $(x_0:x_1) \times (y_0:y_1)$ on $Q$ so that $o_1 = (1:0, 1:0), o_2 = (0:1, 0:1)$ and $ o_3 = (1:1, 1:1)$. Then we observe that the diagonal $L = \lbrace x_0y_1 - x_1y_0 \rbrace$ is tritangent to the the previous curve $C_n$ at $o_1, o_2, o_3$ and that $n \in Q - L$.
Keeping the notations of the previous section we consider the linear system $\mathbf PI$. $C_n$ is in the family of the singular elements of $\mathbf PI$. Let $$ U := Q - L, $$ for each
$n \in U$ we consider the 4-dimensional linear system
 $$
\mathbb D_n \subset \vert \mathcal O_Q(3,3) \vert
$$
of all curves $D$ of bidegree $(3,3)$ such that: \begin{enumerate} \item \it $2o_1 + 2o_2 + 2o_3 \subset L \cdot D$, \item  $D$ has multiplicity $\geq 2$ at $n$, \item $R_i \cdot D = 3n$ for $i =1, 2$, where $R_1$ and $R_2$ are the
lines of $Q$ through $n$. \end{enumerate} 
(i) implies the inclusion $\mathbb D_n \subset \mathbf PI$. We consider the incidence correspondence
$$
\mathbb D := \lbrace (D,n) \in \mathbf PI \times U \ \mid \ D \in \mathbb D_n \rbrace
$$
together with its two projection maps
$$
\begin{CD}
{\mathbf PI } @<{\pi_1}<< {\mathbb D} @>{\pi_2}>> U \\
\end{CD}
$$
$\mathbb D$ is a $\mathbf P^4$-bundle via the map $\pi_2: \mathbb D \to U$. On the other hand the closure of $\pi_1(\mathbb D)$ is the locus of singular elements of 
$\mathbf PI$. Now we define a rational map
$$
m: \mathbb D \to \mathcal S^{1 / 4 -}_3
$$
as follows. Let $C_n \in \mathbb D_n$ be nodal with exactly one node $n$, so that  its normalization $\nu: C \to C_n$ is of genus $3$. Let $\eta := \nu^*\mathcal O_C(o_1+o_2+o_3 - \nu^*n)$, by definition $$ m(C_n) := [C, \eta]. $$ Note that the group $\mathbb G$, defined as in the previous section,  acts on $\mathbb D$ in the natural way. The action of $\alpha \in \mathbb G$ on $\mathbb D$ is the isomorphism
$f_{\alpha}: \mathbb D \to \mathbb D$ sending $(D,n) \in \mathbb D$ to $(\alpha(D), \alpha(n))$. The proof of the next lemma is completely analogous to the proof 
of Lemma \ref{lemma_gen4} and hence we omit it. The corollary is immediate. \begin{lemma} 
Let $D_1, D_2 \in \mathbb F$. Then $m(D_1) = m(D_2)$ iff there exists $\alpha \in \mathbb G$ such that $\alpha(D_1) = \alpha(D_2)$. 
\end{lemma} 
\begin{corollary}\label{cor:gen3}
The quotient $\mathbb D / \mathbb G$ is birational to $\mathcal S^{ 1 / 4 -}_3$. 
\end{corollary}
Finally we can deduce that
\begin{theorem} $\mathcal S^{ 1 / 4-}_3$ is rational. \end{theorem}
\begin{proof}  It is easy to see, and it follows from the analysis of the previous section on the action of $\mathbb G$ on $\mathbf PI$, that the action of $\mathbb G$ on $\mathbb D$ is faithful and linear between the fibres of $\mathbb D$. Hence the $\mathbf P^4$-bundle $\pi_2: \mathbb D \to U$ descends to a $\mathbf P^4$-bundle $\overline {\mathbb D} \to U / \mathbb G$,
which is just $\mathbb D / \mathbb G$.  But $U/\mathbb G$ is rational, since it is a unirational surface, therefore $\overline {\mathbb D} = \mathbb D / \mathbb G$ is rational. Then, by the previous corollary, $\mathcal S^{ 1 / 4-}_3$ is rational.
\end{proof}

\section{The rationality of \texorpdfstring{$\mathcal{S}^{1/4+}_3$}{the moduli space of even 4-roots of genus 3}}

Let us recall that, for any smooth curve $C$ and any divisor $e$ of degree two on it, the line bundle $\omega_C(e)$ is very ample iff $h^0 (\mathcal O_C(e)) = 0$. Let $C$ be a general curve of genus 3 and let $\eta$ be any 4-th root of $\omega_C$. Then $\eta^{\otimes 2}$ is an even theta characteristic. We have considered the case where $\eta^{\otimes 2}$ is odd in the previous section. \par
From now on we assume that $[C, \eta]$ is in $\mathcal S^{1 / 4 +}$, so that $h^0(\eta^{\otimes 2}) = 0$. Then the line bundle $\omega_C \otimes \eta^{\otimes 2}$ is very ample and moreover it defines an embedding
of $C$ in $\mathbf P^3$ as a projectively normal curve whose ideal is generated by cubics, see \cite[\S 6.3]{dolgachev2010topics}. Obviously no quadric contains $C$ and we cannot argue as in the previous section. 
Though the beautiful geometry 
of cubic surfaces through $C$ can be used, it is simpler to consider the canonical model of $C$. Hence we assume that $C$ is embedded in $\mathbf P^2$ as a general plane quartic.
\begin{lemma} \ \par  \begin{enumerate} \it
\item One has $h^0(\eta^{\otimes 3}) = 1$. Moreover, the unique divisor of $\vert \eta^{\otimes 3} \vert$ is supported on three distinct points $o_1, o_2, o_3$. 
\item There exists  exactly one cubic $E$  such  that $C \cdot E =4(o_1 + o_2 + o_3)$. Moreover $E$ is smooth with general moduli.
\end{enumerate}
\end{lemma}
\begin{proof}  
We have $h^0(\eta^{\otimes 3}) \geq 2$ iff $h^0(\omega_C \otimes \eta^{-\otimes 3}) = 1$. This implies that $\omega_C \otimes \eta^{-\otimes 3} \cong \mathcal O_C(p)$,
for some point $p \in C$ such that $4p \in \vert \omega_C \vert$.  But then $C$ is not a general curve. 
To complete the proof of (1) and to prove (2) it suffices to construct a pair $(C, \eta)$ with the required properties. Starting from a smooth cubic $E$ consider three distinct
non collinear points $o_1, o_2, o_3$ such that $b := 4(o_1+o_2+o_3) \in \vert \mathcal O_E(4) \vert$. It is standard to check that the linear system of plane quartics with
base locus $b$ contains a smooth element $C$: see the analogous argument in the proof of theorem 2.3. Let $\eta := \omega_C(-o_1-o_2-o_3)$, then $(C, \eta)$ is the
required pair. 
\end{proof}
Furthermore let $H := \mathcal O_E(1)$ and, as above, $4(o_1+o_2+o_3) = E \cdot C$. Let
$$
\epsilon : = H(-o_1-o_2-o_3).
$$
Clearly $\epsilon$ is a 4-th root of $\mathcal O_E$. Moreover:
\begin{lemma}  
The line bundle $\epsilon^{\otimes 2}$ is not trivial.
\end{lemma}
\begin{proof} 
Assume $\epsilon^{\otimes 2}$ is trivial. Then it follows $2o_1 + 2o_2 + 2o_3 = B \cdot E$, where $B$ is a conic. 
This implies that $h^0(\eta^{\otimes 2}) = h^0(\mathcal O_C(B - 2o_1 - 2o_2 - 2o_3)) = 1$. Hence $\eta^{\otimes 2}$ is an 
odd theta: a contradiction.
\end{proof}
Let $d := o_1 + o_2 + o_3 \in \vert H \otimes \epsilon^{-1} \vert$ be general, it follows from lemma 7.1 and its proof that the linear system
$$
\vert \mathcal I_{4d}(4) \vert
$$
defines a 3-dimensional family of smooth genus 3 spin curves  $(D, \eta_D)$ of order 4, such that $\eta^{\otimes 2}_D$ is an even theta characteristic. Such 
a family is the family of pairs $(D, \eta_D)$ such that $D$ is a smooth element of $\vert \mathcal I_{4d}(4) \vert$ and $\eta_D = \omega_D(-d)$. \par 
Note that the curves $D$ are general in moduli. Since  $\mathcal S^{ 1 / 4 +}_3$ is irreducible and dominates $\mathcal M_3$, it follows that a dense open set of it is 
filled up by points $[D, \eta_D]$ realized as above. We can now use these remarks 
to prove that  $\mathcal S^{1 / 4 +}_3$ is birational to a suitable
tower of projective bundles over a rational modular curve. \par To this purpose we consider the moduli space $\mathcal T$ of triples $(E, H, \tau)$ such that $E$ is an elliptic curve, that is a genus $1$ curve $1$-pointed by $o$, $H = \mathcal O_E(3o)$ and $\tau \in Pic^0 E$ is a 4-torsion point whose square is not trivial.  We then have: 
 \begin{proposition} $\mathcal T$ is a rational curve. \end{proposition}
\begin{proof} Observe that, on a smooth plane cubic $E$, a 4-th root of $ \mathcal O_E$ is a line bundle $\tau := \mathcal O_E(t-o)$ such that
$o, t \in E$ and moreover \\ (i) $3o \in \vert \mathcal O_E(1) \vert$, \\ (ii) $4t + 2o \in \vert \mathcal O_E(2) \vert$,   \\ Indeed these conditions are just
equivalent to say that $4t \sim 4o$. Notice also that they are fulfilled iff there exists a conic $B$ such that $B \cdot E = 4t + 2o$. Furthermore, it is easy to see that either
$\tau^{\otimes 2}$ is not trivial and $B$ is smooth or $B$ is a double line and $B \cdot E = 2(2t + o)$. Assuming the former case we consider
the plane cubic $A + B$, where $A$ is the flex tangent
to $E$ at $o$. Let $P$ be the pencil of cubics generated by $E$ and $A + B$, then its base locus is the 0-dimensional scheme 
$Z := 4t + 5o \subset E$. Note that $Z$, hence $P$, is unique up to projective equivalence. Let $F \in P$ be  smooth, then $F$ is endowed with the line bundles $\tau_F := \mathcal O_F(t - o)$ and $H_F := \mathcal O_F(1)$. Consider the rational map $m: P \to \mathcal T$ defined as follows: $m(F) = [F, H_F, \tau_F]$. The construction implies that $m$ is surjective. Hence $\mathcal T$ is rational.
\end{proof}
Now consider the moduli space $\mathcal A_1(3)$ of abelian curves endowed with a degree $3$ polarization. This is just the moduli space of pairs $(E,H)$. Therefore the curve $\mathcal T$ is a finite cover of $\mathcal A_1(3)$ via the forgetful map
$$ f: \mathcal T \to \mathcal A_1(3), $$ 
sending $[E, H, \tau]$ to $[E,H]$. Over suitable open sets we fix the universal family of abelian curves $\mathcal E \to \mathcal A_1(3)$ and a Poincar\'e sheaf
$\mathcal P$ on $\mathcal A_1(3) \times_{\mathcal M_1} \mathcal E$. Then the restriction of $\mathcal P$ to the curve $[E,H] \times E$ is the line bundle $H$. We  consider the map
$$
f \times id_{\mathcal E}: \mathcal T \times_{\mathcal M_1} \mathcal E \to \mathcal A_1(3) \times_{\mathcal M_1} \mathcal E
$$
and the pull-back $$ \tilde {\mathcal P} := (f \times id_{\mathcal E})^* \mathcal P$$ of $\mathcal P$ over the surface $$ \tilde {\mathcal E} := \mathcal T \times_{\mathcal M_1} \mathcal E. $$ The projection $u: \tilde {\mathcal E} \to \mathcal T$ is an elliptic fibration: its fibre at $[E, H, \tau]$ is the elliptic curve $E = [E,H, \tau] \times E$. Since $E$ is 1-pointed by $o$, the
map $u$ has two sections $$ s_0, s_1: \mathcal T \to \tilde {\mathcal E} $$
which are defined as follows. $s_0$ is the zero section sending $[E, H, \tau]$ to $o$. On the other hand we define $t := s_1([E,H, \tau])$ by the condition $\mathcal O_E(t-o) \cong \tau$. Let 
$$ D_0 := s_0(\mathcal T) \ , \ D_1 := s_1(\mathcal T). $$ Over a dense open set of $\mathcal T$ we can finally define the $\mathbf P^2$-bundles: \par
\begin{itemize} 
\item[$\circ$] $\mathbb T := \mathbf P (u_* \tilde {\mathcal P} \otimes \mathcal O_{\tilde {\mathcal E}}(D_1-D_0))$,
\item[$\circ$]$\mathbb P := \mathbf P (u_* \tilde {\mathcal P}^*)$
\end{itemize} \par
 The fibre of $\mathbb T$ at the point $[E, H, \tau]$ is the linear system $\vert H \otimes \tau \vert$, while the fibre of $\mathbb P$ at the same point is $\mathbf PH^0(H)^*$.
Now we consider the tautological embedding
$$
\tilde {\mathcal E} \subset \mathbb P.
$$
We note that  the embedding $\tilde {\mathcal E}_e \subset \mathbb P_e$ at $e := [E, H, t] \in \mathcal T$,  is  the embedding $E \subset  \mathbf PH^0(H)^*$ defined by $H$. Then we consider the incidence correspondence  
$$
\mathcal Z \subset \mathbb F := \mathbb T \times_{\mathcal T} \mathbb P
$$
parametrizing the points $[E, H, \tau; d, x]  \in \mathbb T \times_{\mathcal T} \mathbb P$ such that \begin{itemize} \item[$\circ$] $ x \in d \subset E \subset \mathbf PH^0(H)^*$, \item[$\circ$] $d \in \vert H \otimes \tau \vert$.
\end{itemize}  Let $\pi_1: \mathbb F \to \mathbb T$ and $\pi_2: \mathbb F \to \mathbb P$ be the projection maps, it is clear that
$$
\mathcal Z \subset \ \pi^*_1 \tilde {\mathcal E} \subset \mathbb F.
$$
Actually $\mathcal Z$ is a divisor in $\pi^*_1 \tilde {\mathcal E}$ and the latter, up to shrinking its base, 
is a smooth family of elliptic curves. Then $4\mathcal Z$ is a Cartier divisor in $\pi_1^* \tilde {\mathcal E}$ and a subscheme of $\mathbb F$. Let $\mathcal J$ be its ideal sheaf, from it we obtain a projective bundle
$$
\mathbb Q := \mathbf P \pi_{1*} (\mathcal J \otimes \pi^*_2 \mathcal O_{\mathbb P}(4)),
$$
over a dense open set of $\mathbb T$. Indeed let $p := [E, H, \tau, d]$ be a general point of $\mathbb T$ and let $\mathcal I_{4d}$ be the ideal sheaf of $4d$ in $ \mathbf P H^0(H)^*$. Then
$H^0(\mathcal I_{4d}(4))$ has constant dimension $4$ and $\mathbb Q$ is a $\mathbf P^3$-bundle over $\mathbb T$ by Grauert's theorem. Moreover $\mathbb Q$ is a $\mathbf P^3$-bundle over $\mathbb T$, which is a $\mathbf P^2$-bundle over the rational curve $\mathcal T$. Hence $\mathbb Q$ is rational. The conclusion is near: we are going to construct a birational map  
$$
m: \mathbb Q \to \mathcal S^{ 1 / 4+}_3. 
$$
Let us define $m$: a general point of $\mathbb Q$ is a general pair $(p,D)$, where $p \in \mathbb T$ is a point as above and $D \in \mathbb Q_p = \vert \mathcal I_{4d}(4) \vert$.
By definition $m(p)$ is the point $[D, \omega_D(-d)]$ of $\mathcal S^{1 / 4}_3$. We conclude that: \par
\begin{theorem} ${\mathcal S}^{1/4+}_3$ is rational. \end{theorem}
\begin{proof} Both $\mathbb Q$ and $\mathcal S^{1 / 4+}_3$ are irreducible of the same dimension. Hence it is enough to show that $m$ is invertible.  Let $[C, \eta] \in \mathcal S^{1/4+}_3$, where $C \subset \mathbf P^2$ is a general smooth quartic and $\eta^{\otimes 3} \cong \mathcal O_C(d)$, as 
above. We know that there exists a unique cubic $E$ such that $E \cdot C = 4d$ and $E$ is general. This defines the point $p = [E, H, \tau, d] \in \mathbb T$, where 
$H := \mathcal O_E(1)$ and $\tau := H(-d)$. Moreover $C$ belongs to $\mathbb Q_p = \vert \mathcal I_{4d}(4) \vert$. Assume $m$ is not invertible at $[C, \eta]$. Then
there exists $D \in \mathbb Q_p$ such that $[D, \omega_D(-d)] = [C, \eta]$ and $D \neq C$. But then there exists a linear isomorphism $\alpha: \mathbf P^2 \to \mathbf P^2$
such that $\alpha(D) = C$ and $\alpha^*\mathcal O_C(d) = \mathcal O_D(d)$. Since $h^0(\eta^{\otimes 3}) = 1$, it follows $\alpha^*(4d) = 4d$ and $\alpha(E) = E$. Since
$E$ is general, $\alpha$ induces a translation or $\pm 1$ multiplication on $\Pic^0 E$. But we have $\alpha^* \tau = \tau$ and moreover $\tau^{\otimes 2}$ is not trivial. This
implies that $\alpha$ is the identity and $D = C$: a contradiction. 
   \end{proof}

\end{document}